\documentclass[11pt]{amsart}
\usepackage{tabularx,booktabs}
\usepackage{caption}
\usepackage{amsmath}
\usepackage{amsfonts}

\usepackage{amscd}
\usepackage{amsthm}
\usepackage{amssymb} \usepackage{latexsym}
\usepackage{eufrak}
\usepackage{euscript}
\usepackage{epsfig}
\usepackage{graphics}
\usepackage{array}
\usepackage{enumerate}
\usepackage{color}
\usepackage{wasysym}
\usepackage{hyperref}
\usepackage{pdfsync}

\usepackage{graphicx}
\usepackage{float}
\usepackage{subfigure}

\newcommand{\bel}[1]{\begin{equation}\label{#1}}

\newcommand{\be}{\begin{equation}}

\newcommand{\ba}{\begin{eqnarray}}
\newcommand{\ea}{\end{eqnarray}}

\newcommand{\qe}{\end{equation}}
\newcommand{\R}{\mathbb{R}}
\newcommand{\N}{\mathbb{N}}
\newcommand{\Z}{\mathbb{Z}}

\newcommand{\Deg}{\mathrm{Deg}}

\newcommand{\diam}{\mathrm{diam}}

\newcommand{\Hmm}[1]{\leavevmode{\marginpar{\tiny%
$\hbox to 0mm{\hspace*{-0.5mm}$\leftarrow$\hss}%
\vcenter{\vrule depth 0.1mm height 0.1mm width \the\marginparwidth}%
\hbox to
0mm{\hss$\rightarrow$\hspace*{-0.5mm}}$\\\relax\raggedright #1}}}

\newtheorem{theorem}{Theorem}[section]

\newtheorem{lemma}[theorem]{Lemma}
\newtheorem{corollary}[theorem]{Corollary}
\newtheorem{definition}[theorem]{Definition}

\newtheorem{remark}[theorem]{Remark}

\begin{document}

\title[Graphs with nonnegative Ricci curvature and maximum degree at most 3]{Graphs with nonnegative Ricci curvature and maximum degree at most 3}

\author{Fengwen Han}
\address{Fengwen, Han: School of Mathematical Sciences, Fudan University, Shanghai 200433, China}
\email{\href{mailto:fwhan19@fudan.edu.cn}{fwhan19@fudan.edu.cn}}

\author{Tao Wang}
\address{Tao, Wang: School of Mathematical Sciences, Fudan University, Shanghai 200433, China}
\email{\href{mailto:taowang21@m.fudan.edu.cn}{taowang21@m.fudan.edu.cn}}


\begin{abstract}
In this paper, we classify graphs with nonnegative Lin-Lu-Yau-Ollivier Ricci curvature, maximum degree at most 3 and diameter at least 6.
\end{abstract}
\maketitle

\section{Introduction}\label{sec:intro}

Ricci curvature is one of the most important geometric quantities in Riemannian geometry. For discrete settings, various definitions of Ricci curvature were developed over the past few decades. In this paper, we focus on Olliver Ricci curvature and its variations. The original Ollivier Ricci curvature was developed on metric spaces and Markov chains, which can be applied on normalized graphs; see \cite{MR2484937,MR2371483}. In 2011, Lin-Lu-Yau \cite{MR2872958} modified Ollivier's definition. This version was widely used in the research of normalized graphs. M\"unch and Wojciechowski \cite{MR3998765} extended Lin-Lu-Yau's definition to general weighted graphs. Bakry-\'{E}mery curvature is another important  Ricci curvature on graphs, readers may refer to \cite{MR1665591,MR2644381} for more details.

There are significant geometric and analytic results for Riemannian manifolds with nonnegative Ricci curvature. As discrete analogs, graphs with nonnegative Ricci curvature are also worthy to study. A basic method is to explore their local geometric structures. Many papers on this subject were published recently. Cushing et al. \cite{cushing2019graph} classified all 3-regular graphs with nonnegative Bakry-\'{E}mery curvature. They also classified all 3-regular graphs with nonnegative original Ollivier Ricci curvature. Under the setting of Lin-Lu-Yau-Ollivier Ricci curvature, there are many works on the classification of Ricci-flat graphs, i.e. graphs whose Ricci curvature vanishes on every edge; see \cite{cushing2018ricci,MR3263934,MR3843277,he2018ricciflat,bai2021ricci}.

It is interesting to study harmonic functions on graphs. For the continuous case, Yau \cite{MR431040} showed that positive harmonic functions on a complete, noncompact Riemannian manifold with nonnegative Ricci curvature are constant, which is known as Liouville property of harmonic functions. For the discrete case, there are also some known results. Jost, M\"unch and Rose \cite{jost2019liouville} proved that bounded harmonic functions are constant under the nonnegative Lin-Lu-Yau-Ollivier Ricci curvature condition. We also refer to \cite{MR3910593,MR3316971,MR4036571} for Liouville properties under the other curvature conditions. However, the strong Liouville property of graphs with non-negative Lin-Lu-Yau-Ollivier Ricci curvature is still unknown. As an application of the main results, we prove that it is true for those graphs with maximum degree at most three.

Let $\mathcal{G}$ be the class of combinatorial graphs with following properties:

1) Lin-Lu-Yau-Ollivier Ricci curvatures on all edges of $G\in\mathcal{G}$ are nonnegative;

2) Degrees of all vertices of $G\in\mathcal{G}$ are at most 3;

3) Diameters of the graphs in $\mathcal{G}$ are at least 6.

The precise definition of $\mathcal{G}$ is given in section 2. We classify all graphs in $\mathcal{G}$ by the following two theorems.
\begin{theorem}
\label{thm-main-finite}
Let $G\in \mathcal{G}$ be a finite graph. Then $G$ is isomorphic to a finite path, a simple cycle, a prism graph, a M\"obius ladder, a particular graph as shown in Figure \ref{fig-main-1}, or a quasi-ladder graph as shown in Figure \ref{fig-main-2}.
\begin{figure}[H]
\centering
\subfigure[Path]{
\includegraphics[width=0.3\textwidth]{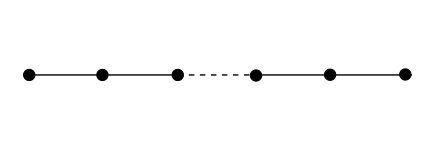}} \quad
\subfigure[Cycle]{
\includegraphics[width=0.2\textwidth]{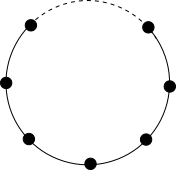}} \quad
\subfigure[Prism graph]{
\includegraphics[width=0.24\textwidth]{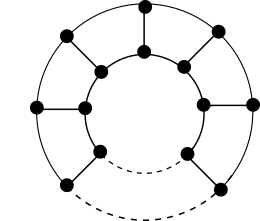}} \quad
\subfigure[M\"obius ladder]{
\includegraphics[width=0.19\textwidth]{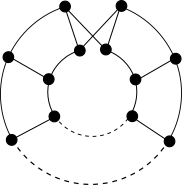}} \qquad
\subfigure[Particular graph]{
\includegraphics[width=0.38\textwidth]{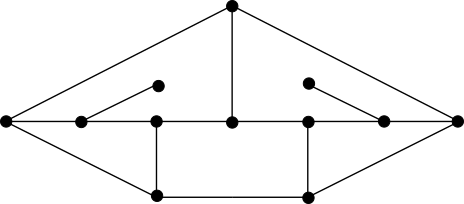}}
\caption{}
\label{fig-main-1}
\end{figure}
\begin{figure}[H]
\centering
\includegraphics[width=0.5\textwidth]{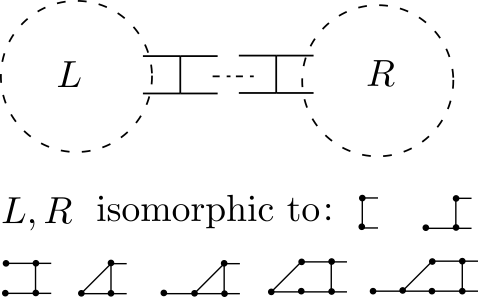}
\caption{Quasi-ladder}
\label{fig-main-2}
\end{figure}
\end{theorem}

\begin{theorem}
\label{thm-main-infty}
Let $G\in \mathcal{G}$ be an infinite graph. Then $G$ is isomorphic to one of the following graphs.
\begin{figure}[H]
\centering
\subfigure[Both-side infinite line]{
\includegraphics[width=0.3\textwidth]{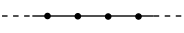}} \ \ 
\subfigure[One-side infinite line]{
\includegraphics[width=0.3\textwidth]{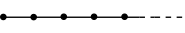}} \ \
\subfigure[Infinite ladder]{
\includegraphics[width=0.3\textwidth]{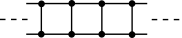}}
\subfigure[Infinite quasi-ladder 1]{
\includegraphics[width=0.3\textwidth]{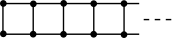}} \ \ 
\subfigure[Infinite quasi-ladder 2]{
\includegraphics[width=0.3\textwidth]{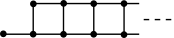}} \ \
\subfigure[Infinite quasi-ladder 3]{
\includegraphics[width=0.3\textwidth]{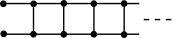}}
\subfigure[Infinite quasi-ladder 4]{
\includegraphics[width=0.3\textwidth]{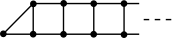}} \ \ 
\subfigure[Infinite quasi-ladder 5]{
\includegraphics[width=0.3\textwidth]{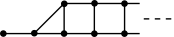}} \ \ 
\subfigure[Infinite quasi-ladder 6]{
\includegraphics[width=0.3\textwidth]{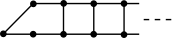}}
\subfigure[Infinite quasi-ladder 7]{
\includegraphics[width=0.3\textwidth]{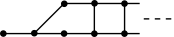}}
\caption{}
\label{fig-main-3}
\end{figure}
\end{theorem}

\begin{remark}
\begin{enumerate}
\item Theorem \ref{thm-main-finite} and Theorem \ref{thm-main-infty} also hold for normalized graphs; see Theorem \ref{thm-2-2}.
\item For normalized graphs, if we change the nonegative Lin-Lu-Yau-Ollivier Ricci curvature constraint to nonnegative Ollivier Ricci curvature, we also get a classification; see Theorem \ref{thm-main-3}.
\item It is well known that manifolds with nonnegative Ricci curvature satisfy Cheeger-Gromoll splitting theorem \cite{MR303460}. As a discrete analog, the splitting theorem also holds for Graphs in $\mathcal{G}$, i.e. if a both-side infinite line is contained in $G\in \mathcal{G}$, then $G$ is the Cartesian product of a line and a graph.
\end{enumerate}
\end{remark}

Bai, Lu and Yau \cite{bai2021ricci} classified all Ricci-flat graphs with maximum degree at most 4. Compared with our results, they require a more stringent curvature condition but relax the maximum degree condition from 3 to 4, which causes a huge increase in the complexity of the problem.

The proofs of the main results are shown in section 4. We give a sketch here. Let $P=v_0v_1\cdots v_l$ be a diameter of $G \in \mathcal{G}$ with $l \ge 6$. The local structure along $P$ can be characterized in Lemma \ref{lem-3-1} , Lemma \ref{lem-3-3} and Lemma \ref{lem-3-2}. Moreover, the subgraph induced by $P$ and its neighbors is a quasi-ladder. We divide $G$ into four parts as shown in Figure \ref{intro}. If $L$ is connected to $R$ via $O$, we will get a special cycle (we call it a geodesic cycle in this paper) by Lemma \ref{lem-4-1}, and we prove $G$ has to be a cycle, a prism or a M\"obius ladder by Lemma \ref{lem-4-2}. If $L$ is not connected to $R$ via $O$, then $O=\emptyset$ and $G$ is a quasi-ladder, otherwise we will find a longer geodesic path. For infinite graphs, the proof is similar. For graphs with small diameter, say $l=4 \text{ or } 5$, the structures are more complicated. However, since the maximum degree is at most 3, graphs with nonnegative curvature and small diameter can be enumerated by computers, which we are not going to explore in this paper.
\begin{figure}[H]
\centering
\includegraphics[width=0.55\textwidth]{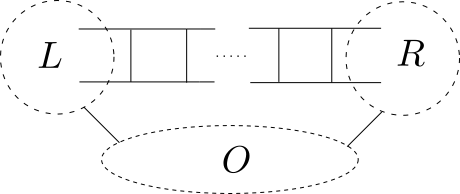}
\caption{}
\label{intro}
\end{figure}
The paper is organized as follows: In next section, we introduce some notions on graphs and the definition of Ricci curvature. Section 3 gives the local structures of graph along a diameter, which serve as basic tools in the following discussions. In section 4, we prove the main results of this paper. In section 5, as an applications of the main results, we show that all positive harmonic functions on $G \in \mathcal{G}$ are constant.


\section{Preliminaries}

\subsection{Basic notions}

Throughout this paper, we always assume that $G(V,E)$ is an undirected, connected, locally finite, simple graph with vertex set $V$ and edge set $E$. Let $|G|:=|V|$ be the number of vertices of the graph. For any vertices $u,v \in V$, we call $v$ a \emph{neighbor} of $u$ if $\{u,v\} \in E$, denoted by $u \sim v$. For notational simplicity, we write $uv$ for the unordered pair $\{u, v\}$. For $A,B \subset V$, let $E(A,B):=\{uv\in E :u\in A,v\in B\}$. The \emph{degree} of a vertex $v\in V$ is defined via $\deg(v):=\#\{u\in V:u\sim v\}$. A graph $H(V',E')$ is called a \emph{subgraph} of $G(V,E)$ if $V' \subseteq V$ and $E' \subseteq E$. A subgraph $H$ is called an \emph{induced subgraph} of $G$ if $E'=\{ \{u,v\} \in E: u,v \in V' \}$, and is denoted by $G(V')$. For $A,B \subseteq V'$, let $d_H(A,B):= \inf \{n: A \ni v_0 \sim \cdots v_n \in B\}$ be the \emph{combinatorial distance} between $A$ and $B$ in subgraph $H(V',E')$. If $H=G$, we omit the subscript, i.e. $d(A,B):=d_G(A,B)$. If $A=\{u\}$ or $B=\{v\}$ is a single vertex subset, we will omit the braces, i.e. $d_H(u,v):=d_H(\{u\},\{v\})$. 

A \emph{walk} $W$ of length $l$ is a sequence of vertices $v_0,v_1,\cdots,v_l$ such that $v_i \sim v_{i+1}$ for $i=0,1,\cdots,l-1$. We write $W=v_0v_1\cdots v_l$ for short. If $v_0=v_l$, we call it a \emph{closed walk}. Moreover, a closed walk  $W=v_0v_1\cdots v_l$ is called a \emph{cycle} if $v_i \ne v_j$ for all $1\leq i,j \leq l$. We also regard a cycle $C=v_0v_1\cdots v_l$ as a subgraph with vertex set $V(C)=\{v_0,v_1 \cdots v_{l-1}\}$ and edge set $E(C)=\{\{v_i,v_{i+1}\}:i=0,1,\cdots l-1\}$. A walk $W=v_0v_1\cdots v_l$ is called a \emph{path} if $v_i \ne v_j$ for all $0\leq i,j \leq l$. We also regard a path $P=v_0v_1\cdots v_l$ as a subgraph with vertex set $V(P)=\{v_0,v_1 \cdots v_l\}$ and edge set $E(P)=\{\{v_i,v_{i+1}\}:i=0,1,\cdots l-1\}$.


Let $H$ be a subgraph of $G$. We say a path $P=v_0v_1\cdots v_l$ is \emph{geodesic} in $H$ if $d_H(v_0,v_l)=l$. It is easy to know that $P=v_0v_1\cdots v_l$ is a geodesic path of $H$ if and only if $d_H(v_i,v_j)=|i-j|$ for all $0\le i,j \le l$. We say $H$ is a \emph{geodesic subgraph} if all geodesic paths in $H$ are also geodesic paths in $G$. It is not hard to verify that $H$ is a geodesic subgraph of $G$ if and only if $d_H(u,v)=d_G(u,v)$ for all $u,v \in H$. A \emph{diameter path} of $G$ is a longest geodesic path $P$ in $G$. We let $\diam(G)$ denote the length of a diameter path in $G$, i.e. $\diam(G)=\sup\{d(u,v):u,v\in V\}$.

A graph $G(V,E)$ is called \emph{rooted at} $p$ if a vertex $p\in V$ is specially appointed, denoted by $(G,p)$. We define $r: V \to \N^0$ via $r(u) = d(p,u)$. If $r(u)>0$, there always exists a $v\sim u$ such that $r(v)=r(u)-1$ and a geodesic path $p\cdots v u$ with length equals to $r(u)$. Given a geodesic path $P=v_0v_1\cdots v_l$ in $G$, we always regard $G$ as a graph rooted at $v_0$ if not specified.


Let $G(V,E)$ be a graph with degree at most 3, $P=v_0v_1\cdots v_l$ be a geodesic path. For all $1\le i \le l-1$, $v_i$ has at most one extra neighbor as $v_{i-1}$ and $v_{i+1}$ are two neighbors of $v_i$. We define a \emph{state function} $s:\{v_i:1\le i \le l-1\} \to \{2,3^-,3^0,3^+\}$ via
\begin{align*}
\begin{split}
s(v_i)=
\begin{cases}
2, \ \ \ \ \ \ \ \ \  & \text{if} \deg(v_i)=2,\\
3^-,  & \text{if } v_i \text{ has an extra neighbor }u_i\text{ and }r(u_i)<r(v_i),\\
3^0,  & \text{if } v_i \text{ has an extra neighbor }u_i\text{ and }r(u_i)=r(v_i),\\
3^+,  & \text{if } v_i \text{ has an extra neighbor }u_i\text{ and }r(u_i)>r(v_i).
\end{cases}
\end{split}
\end{align*}

As figures are quite important in the discussions of this paper, we explain the drawing style here. If not specified, a horizontal segment with vertices represents a geodesic path as shown in (a) of the following figure. If $v_i$ has an extra neighbor $u_i$ and  $r(u_i)<r(v_i)$, i.e. $s(v_i)=3^-$, we draw the figure as shown in (b) such that $u_i$ has the same abscissa with $v_{i-1}$ as $r(u_i)=r(v_{i-1})$. Figure (c) and (d) show the drawings when $r(u_i)=r(v_i)$ and $r(u_i)>r(v_i)$ respectively.

\begin{figure}[htp]
\centering
\subfigure[]{
\includegraphics[width=0.3\textwidth]{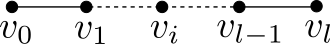}}
\qquad
\subfigure[]{
\includegraphics[width=0.3\textwidth]{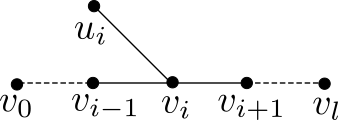}}
\subfigure[]{
\includegraphics[width=0.3\textwidth]{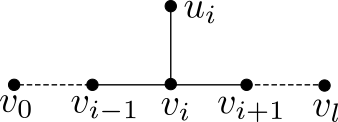}}
\qquad
\subfigure[]{
\includegraphics[width=0.3\textwidth]{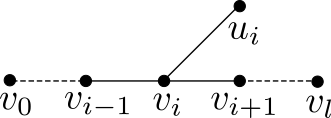}}
\caption{}
\end{figure}

\subsection{Ricci curvature on weighted graphs}

Let
$$w: E\to \mathbb (0,\infty),\ \{x,y\} \mapsto w(x,y)=w(y,x)$$
be the \emph{edge weight} function,
$$m:V\to (0,\infty),\ x\mapsto m(x)$$
be the \emph{vertex weight} function. We call the quadruple $G=(V, E, m,w)$ a \emph{weighted graph}. Unless otherwise specified, we always regard a graph $G(V,E)$ as a \emph{combinatorial graph}, i.e. $w \equiv 1$ and $m \equiv 1$. Our results also hold for \emph{normalized graphs}, i.e. $w \equiv 1$ and $m(u) = \deg(u)$ for all $u \in V$; see Theorem \ref{thm-2-2}.

We define the \emph{Laplacian} $\Delta: \R^V \to \R^V$ via\\
$$\Delta f(u):=\frac{1}{m(u)}\sum_{v\sim u}w(u,v)(f(v)-f(u)).$$
A function $f$ on $G$ is called \emph{harmonic} if $\Delta f \equiv 0$.

For a weighted graph $G(V,E,w,m)$, let $\Deg(u):=\sum_{v\sim u} \frac{w(u,v)}{m(u)}$. We define a probability measure
\begin{align*}
\begin{split}
\mu_u^\varepsilon(v)=
\begin{cases}
1-\varepsilon \Deg(u) \ \ \ \ \ \ \ &:v=u;\\
\varepsilon w(u,v)/m(u)\ &:\text{otherwise}.
\end{cases}
\end{split}
\end{align*}
Particularly, for combinatorial graphs, we have
\begin{align*}
\begin{split}
\mu_u^\varepsilon(v)=
\begin{cases}
1-\varepsilon \deg(u) \ \ \ \ \ \ \ &:v=u;\\
\varepsilon &:\text{otherwise}.
\end{cases}
\end{split}
\end{align*}
For normalized graphs, we have
\begin{align*}
\begin{split}
\mu_u^\varepsilon(v)=
\begin{cases}
1-\varepsilon \ \ \ \ \ \ \ \ \ \ &:v=u;\\
\varepsilon /\deg(u)\ &:\text{otherwise}.
\end{cases}
\end{split}
\end{align*}

\begin{definition}(Wasserstein distance)
Let $G(V,E)$ be a graph, $\mu$ and $\nu$ be two probability measures on $G$. The Wasserstein distance $W(\mu,\nu)$ is given by
$$W(\mu,\nu):=\inf_A\sum_{u,v\in V}A(u,v)d(u,v),$$
where $A$ is a coupling between $\mu$ and $\nu$, i.e. a mapping $A:V\times V \to [0,1]$ s.t.
$$\sum_{v\in V}A(u,v)=\mu (u) \text{ and } \sum_{u\in V}A(u,v)=\nu (v).$$
\end{definition}

\begin{definition}
For any vertices $u,v \in G(V,E,w,m)$, the $\varepsilon$-Ollivier-Ricci curvature $\kappa_{\varepsilon}(u,v)$ is defined via
$$\kappa_{\varepsilon}(u,v):=1-\frac{W(\mu_u^{\varepsilon},\mu_v^{\varepsilon})}{d(u,v)}.$$
Specifically, the Ollivier Ricci curvature on a normalized graph is defined via
$$\kappa^O(u,v):=\kappa_1(u,v).$$
The Lin-Lu-Yau-Ollivier Ricci curvature $\kappa(u,v)$ is defined via
$$\kappa(u,v):=\lim_{\varepsilon \to 0^+} \frac{1}{\varepsilon}\kappa_{\varepsilon}(u,v).$$
\end{definition}

For fixed vertices $u,v\in E$, it was shown in \cite{MR3815539} that the function $\varepsilon \mapsto \kappa_{\varepsilon}(u,v)$ is concave and piecewise linear, so the limitation exists and $\kappa$ is well defined. There are also equivalent limit-free definitions via coupling function or Laplacian on graphs; see \cite{bai2020sum,MR3998765}. When we say Ricci curvature or curvature in the following, we always mean the Lin-Lu-Yau-Ollivier Ricci curvature defined above.

For combinatorial graphs, M\"unch and Wojciechowsk \cite[Theorem 2.6]{MR3998765} gave an easier way to calculate Ricci curvatures on edges as shown in the following theorem.
\begin{theorem}
\label{thm-tnc}
(Transport and combinatorial graphs) Let $G=(V,E)$ be a combinatorial graph and let $u \sim v$ be adjacent vertices. Let $B_{uv}:=B_1(u) \cap B_1(v)$, $B_u^v:=B_1(u) \setminus B_1(v)$ and $B_v^u:=B_1(v) \setminus B_1(u)$. Let $\Phi _{uv}:=\{\phi:D(\phi) \subseteq B_u^v \to R(\phi) \subseteq B_v^u\ :\phi \ bijective\}$. For $\phi \in \Phi_{uv}$ write $D(\phi)^c:=B_u^v\setminus D(\phi)$ and $R(\phi)^c:=B_v^u\setminus R(\phi)$. Then
$$\kappa(u,v)=\# B_{uv} - \inf_{\phi \in \Phi_{uv}}\left( \# D(\phi)^c + \#R(\phi)^c + \sum_{w\in D(\phi)}[d(w,\phi (w))-1] \right).$$
\end{theorem}
\begin{figure}[H]
\centering
\includegraphics[width=0.7\textwidth]{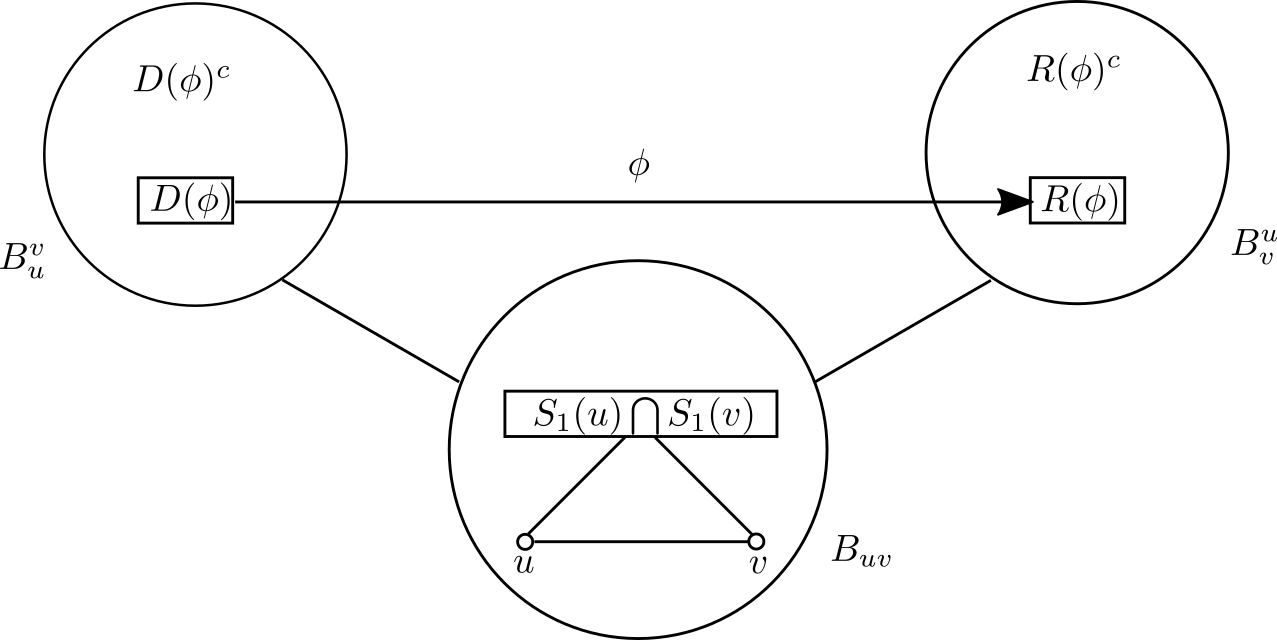}
\caption{}
\end{figure}

We always suppose $\# D(\phi)^c=0$ or $\#R(\phi)^c=0$ when applying Theorem \ref{thm-tnc}. If $\phi$ is an optimal map such that $u\in D(\phi)^c$ and $v \in R(\phi)^c$, we extend $\phi$ by adding an extra map $\phi(u)=v$. It is easy to verify that the new $\phi$ is still an optimal map. So our hypothesis is fair. The following is a direct corollary of Theorem \ref{thm-tnc}.

\begin{corollary}
\label{cor-tnc}
Let $G=(V,E)$ be a combinatorial graph and let $u \sim v$ be adjacent vertices. Let $N_w:=B_1(w) \setminus \{u,v\}$ be the set of extra neighbors of $w$, $w\in \{u,v\}$. If $|N_u|+|N_v|\ge 3$ and $d(N_u,N_v) \ge 3$, we have $\kappa(u,v) < 0$.
\end{corollary}

Let $\kappa(G):=\inf_{\{u,v\}\in E}\kappa(u,v)$ be the Ricci curvature of $G$ and $\deg(G):= \sup_{u \in V}\deg(u)$ the maximum degree of $G$. The class of graphs we will classify in this paper is defined by
$$\mathcal{G}:=\{G(V,E): \deg(G)\le 3,\kappa(G)\ge 0, \diam(G)\ge 6\}.$$

Cushing et al. \cite{cushing2019graph} provide a powerful online tool, the Graph Curvature Calculator, to calculate various of curvatures on graphs. Readers can find it on https://www.mas.ncl.ac.uk/graph-curvature/.  By using this tool, it is convenient to check that the Ricci curvatures of graphs in Figure \ref{fig-main-1}, \ref{fig-main-2} and \ref{fig-main-3} are nonnegative. We also use it to prove the following theorems.

\begin{theorem}
\label{thm-2-2}
Let $G=(V,E)$ be a graph with $\deg(G)\le 3$. Let $G_C$ be the combinatorial graph and $G_N$ be the normalized graph obtained by equipping $G$ with combinatorial weight and normalized weight respectively. Then $\kappa(G_C)\ge 0$ if and only if $\kappa(G_N)\ge 0$.
\end{theorem}
\begin{proof}
For any $uv \in E$, let $\kappa_C(uv)$ be the curvature of $uv$ when $G$ is equipped with combinatorial weight and $\kappa_N(uv)$ be the curvature of $uv$ when $G$ is equipped with normalized weight. If $\deg(u)=1$ or $\deg(v)=1$, we have $\kappa_C(uv)\ge 0$ and $\kappa_N(uv)\ge 0$ by a direct calculation. If $\deg(u)=\deg(v)$, we have $\kappa_C(uv)=\deg(u) \cdot \kappa_N(uv)$ by the definition of Ricci curvature. So we only need to consider the case that $\deg(u)=2$ and $\deg(v)=3$, i.e. $u$ has an extra neighbor $u'$, $v$ has two extra neighbors $v'$ and $v''$. Then $\kappa_C(uv)$ and $\kappa_N(uv)$ depend only on the distances $d(u',v')$ and $d(u',v'')$. By a direct calculation, we have $\kappa_C(uv) < 0$ $\Leftrightarrow$ $d(u',v')=d(u',v'')=3$ $\Leftrightarrow$ $\kappa_N(uv) < 0$. This completes the proof.
\end{proof}

It is noteworthy that Theorem \ref{thm-2-2} does not hold for graphs with maximum degree larger than 3. The minimum Ricci curvatures of the following graphs are 0, 0.5, -0.133 as normalized graphs and -1, -1, 0 as combinatorial graphs respectively. \footnote{Graph (a) built by Bai-Lu-Yau \cite{bai2021ricci} is Ricci-flat as a normalized graph. Graphs (b) and (c) are pointed out to us by Florentin M\"unch.}
\begin{figure}[H]
\centering
\subfigure[]{
\includegraphics[width=0.2\textwidth]{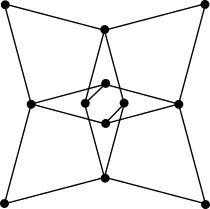}}
\qquad
\subfigure[]{
\includegraphics[width=0.2\textwidth]{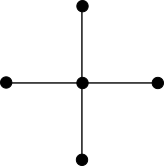}}
\qquad
\subfigure[]{
\includegraphics[width=0.2\textwidth]{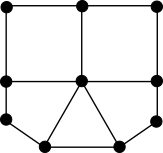}}
\caption{}
\end{figure}

\begin{theorem}
\label{thm-main-3}
Let $G$ be a normalized graph with nonnegative Ollivier Ricci curvature, maximum degree at most 3 and diameter at least 6. If $G$ is finite, then it is isometric to a finite path, a simple cycle, a prism graph, a M\"obius ladder or a quasi-ladder as shown in the following figure. If $G$ is infinite, then it is isometric to (a), (b), (c), (d), (e), (f), (g) or (h) of Figure \ref{fig-main-3}.
\end{theorem}
\begin{figure}[H]
\centering
\includegraphics[width=0.5\textwidth]{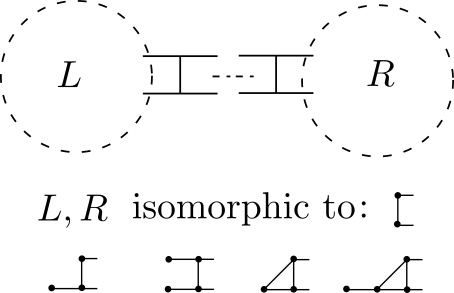}
\caption{Quasi-ladder}
\end{figure}
\begin{proof}
For normalized graphs, $\kappa_0(u,v)\equiv 0$ for all $u,v \in V$. If $\kappa^O(u,v)\ge 0$, we always have $\kappa(u,v)\ge 0$, as the function $\varepsilon \mapsto \kappa_{\varepsilon}(u,v)$ is concave and piecewise linear; see \cite[Theorem 1.1]{MR3815539}. This indicates that all normalized graphs with nonnegative Ollivier Ricci curvature are also graphs with nonnegative Lin-Lu-Yau-Ollivier Ricci curvature. We check the Ollivier Ricci curvatures of graphs in Figure \ref{fig-main-1}, \ref{fig-main-2}, \ref{fig-main-3} by using the Graph Curvature Calculator and get the result.
\end{proof}



\section{The local structures}
In this section, we study the local structure on a geodesic path of a graph $G \in \mathcal{G}$.

\begin{lemma}
\label{lem-3-1}
Let $P=v_0 \cdots v_iv_{i+1}v_{i+2}v_{i+3} \cdots v_l$ be a geodesic path of $G\in \mathcal{G}$.

Case $(2,2)$: If $(s(v_{i+1}),s(v_{i+2}))=(2,2)$, then $G$ contains the following as a subgraph.
\begin{figure}[H]
\centering
\includegraphics[width=0.23\textwidth]{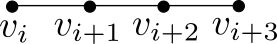}
\caption{Case $(2,2)$}
\label{fig-ls22}
\end{figure}

Case $(2,3^-)$: If $(s(v_{i+1}),s(v_{i+2}))=(2,3^-)$, then $G$ contains one of the followings as a subgraph.
\begin{figure}[H]
\centering
\subfigure[]{
\includegraphics[width=0.23\textwidth]{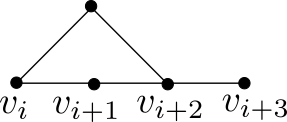}} \quad
 \
\subfigure[]{
\includegraphics[width=0.23\textwidth]{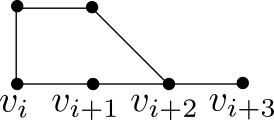}} \quad
 \
\subfigure[]{
\includegraphics[width=0.23\textwidth]{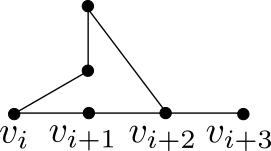}}
\caption{Case $(2,3^-)$}
\label{fig-ls23-}
\end{figure}

Case $(2,3^0)$: If $(s(v_{i+1}),s(v_{i+2}))=(2,3^0)$, then $G$ contains the following as a subgraph.
\begin{figure}[H]
\centering
\includegraphics[width=0.23\textwidth]{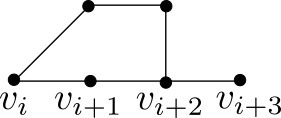}
\caption{Case $(2,3^0)$}
\label{fig-ls230}
\end{figure}

Case $(2,3^+)$: The situation $(s(v_{i+1}),s(v_{i+2}))=(2,3^+)$ can not exist.

Case $(3^-,2)$: The situation $(s(v_{i+1}),s(v_{i+2}))=(3^-,2)$ can not exist.

Case $(3^0,2)$: If $(s(v_{i+1}),s(v_{i+2}))=(3^0,2)$, then $G$ contains the following as a subgraph.
\begin{figure}[H]
\centering
\includegraphics[width=0.23\textwidth]{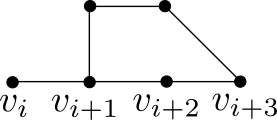}
\caption{Case $(3^0,2)$}
\label{fig-ls302}
\end{figure}

Case $(3^+,2)$: If $(s(v_{i+1}),s(v_{i+2}))=(3^+,2)$, then $G$ contains one of the followings as a subgraph.
\begin{figure}[H]
\centering
\subfigure[]{
\includegraphics[width=0.23\textwidth]{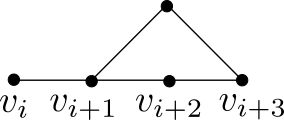}} \quad
 \
\subfigure[]{
\includegraphics[width=0.23\textwidth]{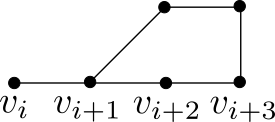}} \quad
 \
\subfigure[]{
\includegraphics[width=0.23\textwidth]{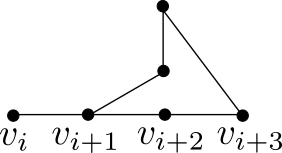}}
\caption{Case $(3^+,2)$}
\label{fig-ls3+2}
\end{figure}

Case $(3^-,3^-)$: If $(s(v_{i+1}),s(v_{i+2}))=(3^-,3^-)$, then $G$ contains one of the followings as a subgraph.
\begin{figure}[H]
\centering
\subfigure[]{
\includegraphics[width=0.23\textwidth]{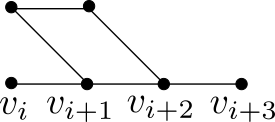}} \quad
\qquad
\subfigure[]{
\includegraphics[width=0.23\textwidth]{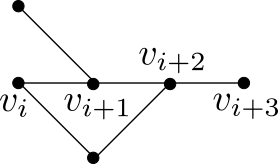}}
\caption{Case $(3^-,3^-)$}
\label{fig-ls3-3-}
\end{figure}

Case $(3^-,3^0)$: The situation $(s(v_{i+1}),s(v_{i+2}))=(3^-,3^0)$ can not exist.

Case $(3^-,3^+)$: The situation $(s(v_{i+1}),s(v_{i+2}))=(3^-,3^+)$ can not exist.

Case $(3^0,3^-)$: If $(s(v_{i+1}),s(v_{i+2}))=(3^0,3^-)$, then $G$ contains one of the followings as a subgraph.
\begin{figure}[H]
\centering
\subfigure[]{
\includegraphics[width=0.23\textwidth]{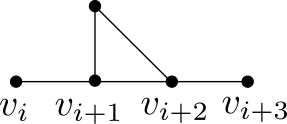}} \ \
\subfigure[]{
\includegraphics[width=0.23\textwidth]{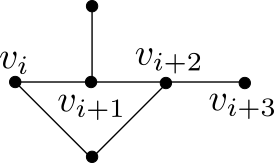}} \ \
\subfigure[]{
\includegraphics[width=0.23\textwidth]{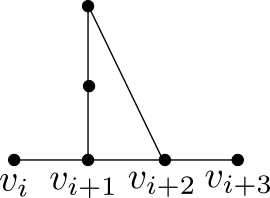}} \ \
\subfigure[]{
\includegraphics[width=0.23\textwidth]{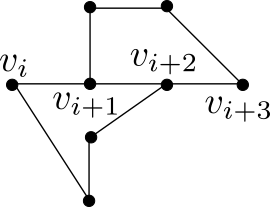}} \ \
\subfigure[]{
\includegraphics[width=0.23\textwidth]{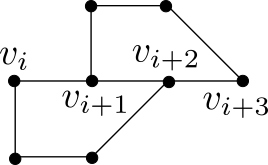}}
\caption{Case $(3^0,3^-)$}
\label{fig-ls303-}
\end{figure}

Case $(3^0,3^0)$: If $(s(v_{i+1}),s(v_{i+2}))=(3^0,3^0)$, then $G$ contains one of the followings as a subgraph.
\begin{figure}[H]
\centering
\subfigure[]{
\includegraphics[width=0.23\textwidth]{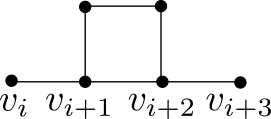}} \quad
\qquad
\subfigure[]{
\includegraphics[width=0.23\textwidth]{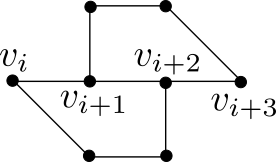}}
\caption{Case $(3^0,3^0)$}
\label{fig-ls3030}
\end{figure}

Case $(3^0,3^+)$: The situation $(s(v_{i+1}),s(v_{i+2}))=(3^0,3^+)$ can not exist.

Case $(3^+,3^-)$: If $(s(v_{i+1}),s(v_{i+2}))=(3^+,3^-)$, then $G$ contains one of the followings as a subgraph.
\begin{figure}[H]
\centering
\subfigure[]{
\includegraphics[width=0.23\textwidth]{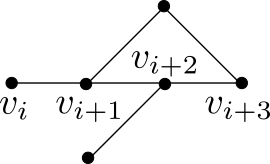}} \ \
\subfigure[]{
\includegraphics[width=0.23\textwidth]{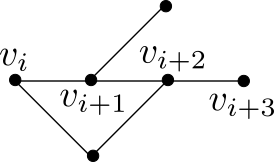}} \ \
\subfigure[]{
\includegraphics[width=0.23\textwidth]{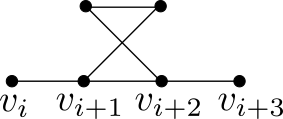}} \ \
\subfigure[]{
\includegraphics[width=0.23\textwidth]{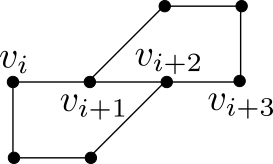}} \ \
\subfigure[]{
\includegraphics[width=0.23\textwidth]{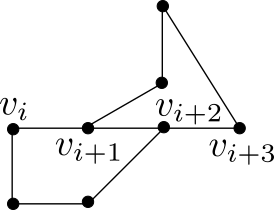}} \ \
\subfigure[]{
\includegraphics[width=0.23\textwidth]{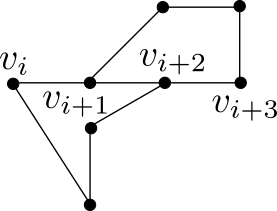}} \ \
\subfigure[]{
\includegraphics[width=0.23\textwidth]{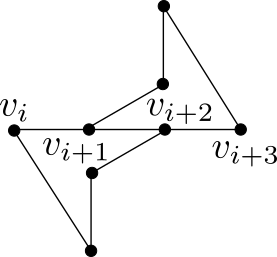}}
\caption{Case $(3^+,3^-)$}
\label{fig-ls3+3-}
\end{figure}

Case $(3^+,3^0)$: If $(s(v_{i+1}),s(v_{i+2}))=(3^+,3^0)$, then $G$ contains one of the followings as a subgraph.
\begin{figure}[H]
\centering
\subfigure[]{
\includegraphics[width=0.23\textwidth]{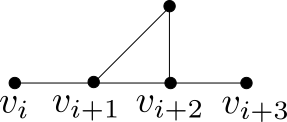}} \ \
\subfigure[]{
\includegraphics[width=0.23\textwidth]{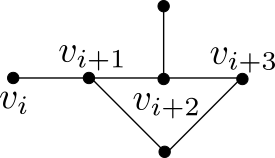}} \ \
\subfigure[]{
\includegraphics[width=0.23\textwidth]{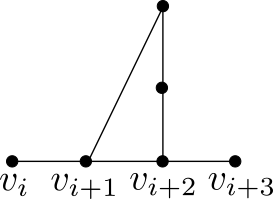}} \ \
\subfigure[]{
\includegraphics[width=0.23\textwidth]{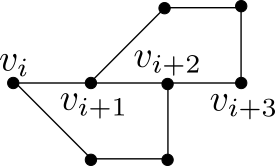}} \ \
\subfigure[]{
\includegraphics[width=0.23\textwidth]{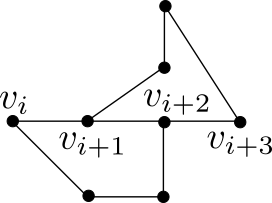}}
\caption{Case $(3^+,3^0)$}
\label{fig-ls3+30}
\end{figure}

Case $(3^+,3^+)$: If $(s(v_{i+1}),s(v_{i+2}))=(3^+,3^+)$, then $G$ contains one of the followings as a subgraph.
\begin{figure}[H]
\centering
\subfigure[]{
\includegraphics[width=0.23\textwidth]{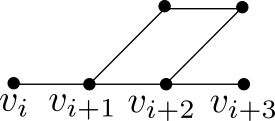}} \quad
\qquad
\subfigure[]{
\includegraphics[width=0.23\textwidth]{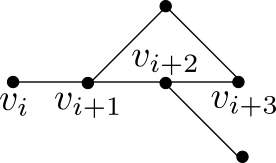}}
\caption{Case $(3^+,3^+)$}
\label{fig-ls3+3+}
\end{figure}
\end{lemma}

\begin{proof}
We only show the proof of case $(2,3^-)$, proofs for other cases are similar. Without loss of generality, let $i=0$. Then $G$ has a subgraph as shown in the following figure with $\deg(v_1)=2$.
\begin{figure}[H]
\centering
\includegraphics[width=0.2\textwidth]{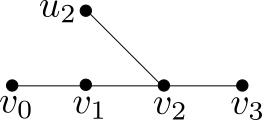}
\caption{}
\end{figure}
By Theorem \ref{thm-tnc} we have
$$\kappa(v_1,v_2)=\# B_{v_1v_2} - \inf_{\phi \in \Phi_{v_1v_2}}(\# D(\phi)^c + \#R(\phi)^c + \sum_{w\in D(\phi)}[d(w,\phi (w))-1]),$$
where $B_{v_1v_2}=\{v_1,v_2\}$, $\Phi_{v_1v_2}=\{ \text{all maps from $\{v_0\}$ to $\{v_3,u_2\}$}\}$, $\# D(\phi)^c=0$, and $\#R(\phi)^c=1$. Then we have
\begin{equation*}
\begin{split}
0 &\le\kappa(v_1,v_2)\\
&=2-\inf_{\phi(v_0)\in \{v_3,u_2\}}(0+1+d(v_0,\phi(v_0))-1)\\
&=\max\{2-d(v_0,v_3),2-d(v_0,u_2)\}.
\end{split}
\end{equation*}
As $d(v_0,v_3)=3$, we have $d(v_0,u_2)=1$ or $d(v_0,u_2)=2$. If $d(v_0,u_2)=1$, $G$ contain (a) of Figure \ref{fig-ls23-} as a subgraph. If $d(v_0,u_2)=2$, $G$ contains (b) or (c) of Figure \ref{fig-ls23-} as a subgraph. This completes the proof.
\end{proof}

By applying case $(\cdot,3^+)$ or case $(3^-,\cdot)$ of Lemma \ref{lem-3-1} repeatedly, we get the following corollary.
\begin{corollary}
\label{cor-3-1}
Let $P=v_0v_1 \cdots v_l$ be a geodesic path in $G\in \mathcal{G}$, $l\ge 4$.

1) If there exists a $ i\in\{1,2,\cdots,l-1\}$ such that $s(v_i)=3^+$, then $s(v_j)=3^+$ for all $1\le j \le i$ and $G$ has a subgraph (b) or (c) of the following figure.

2) If there exists a $i\in\{2,\cdots,l-1\}$ such that $s(v_i)=3^-$, then $s(v_j)=3^-$ for all $i\le j \le l-1$ and $G$ has a subgraph (e) or (f) of the following figure.
\end{corollary}

\begin{figure}[H]
\centering
\subfigure[]{
\includegraphics[width=0.28\textwidth]{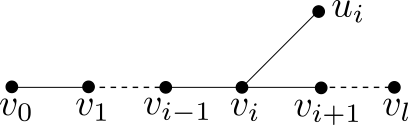}} \ \
\subfigure[]{
\includegraphics[width=0.28\textwidth]{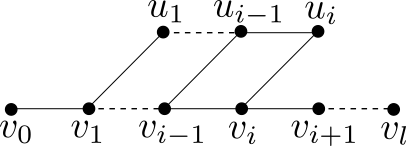}} \ \
\subfigure[]{
\includegraphics[width=0.28\textwidth]{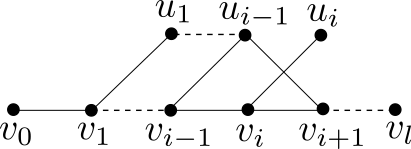}} \ \
\subfigure[]{
\includegraphics[width=0.28\textwidth]{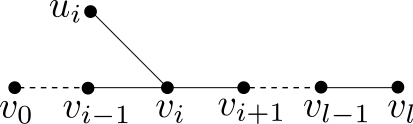}} \ \
\subfigure[]{
\includegraphics[width=0.28\textwidth]{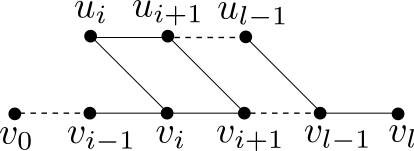}} \ \
\subfigure[]{
\includegraphics[width=0.28\textwidth]{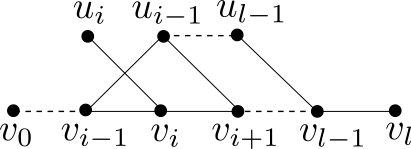}} \ \
\caption{}
\end{figure}

As a supplement to Lemma \ref{lem-3-1}, we explore the local structures on long geodesic paths in the following two lemmas,

\begin{lemma}
\label{lem-3-3}
Let $P=v_0 \cdots v_iv_{i+1}v_{i+2}v_{i+3} \cdots v_l$ be a geodesic path with  $l\ge6$. If $G$ contains (b) of Figure \ref{fig-ls3030} as a subgraph, then $G$ is isometric to the particular graph (e) of Figure \ref{fig-main-1}.
\end{lemma}
\begin{proof}
If $i=2$, $G$ contains a subgraph (a) of the following figure. By Corollary \ref{cor-3-1}, $G$ contains a subgraph (b). By applying case $(3^+,3^0)$ of Lemma \ref{lem-3-1} on $v_1v_2v_3v_4$, $G$ contains a subgraph (c). By applying case $(3^+,3^0)$ of Lemma \ref{lem-3-1} on $v_1u_1u_3u_5$, $G$ contains a subgraph (d). Moreover, $\deg(v_0)=\deg(v_6)=1$ by Corollary \ref{cor-tnc}. So $G$ is isometric to graph (d) of the following figure, which is isometric to (e) of Figure \ref{fig-main-1}. If $i\ne 2$, it is impossible to get a graph with nonnegative Ricci curvature by a similar argument. This proves the lemma.
\begin{figure}[H]
\centering
\subfigure[]{
\includegraphics[width=0.33\textwidth]{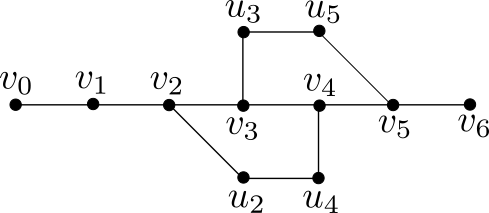}} \qquad
\quad
\subfigure[]{
\includegraphics[width=0.33\textwidth]{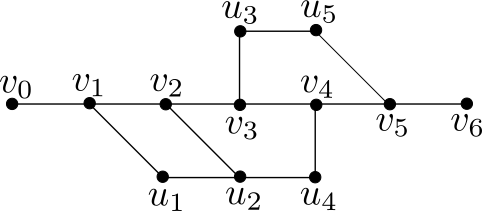}} \qquad
\subfigure[]{
\includegraphics[width=0.33\textwidth]{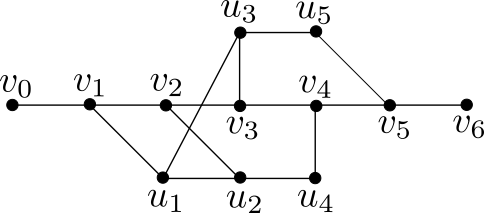}} \qquad
\quad
\subfigure[]{
\includegraphics[width=0.33\textwidth]{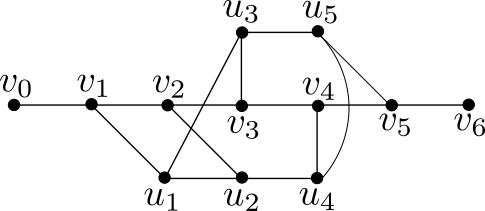}}
\caption{}
\label{fig-lls1}
\end{figure}
\end{proof}

\begin{lemma}
\label{lem-3-2}
Let $P=v_0\cdots v_iv_{i+1}v_{i+2}v_{i+3}\cdots v_l$ be a geodesic path with $l\ge6$. Then none of the followings can be a subgraph of $G$: (c) in Figure \ref{fig-ls23-}; (c) in Figure \ref{fig-ls3+2}; (c), (d) and (e) in Figure \ref{fig-ls303-}; (c), (d), (e), (f) and (g) in Figure \ref{fig-ls3+3-}; (c), (d) and (e) in Figure \ref{fig-ls3+30}.
\end{lemma}

\begin{proof}
We only show (c) of Figure \ref{fig-ls3+2} can not be a subgraph of $G$. Proofs for other cases are similar. We prove by contradiction. Suppose $i \le 1$, then $i+5 \le l$ and $G$ contains a subgraph (a) of the following figure with $\deg(v_{i+2})=2$ by Corollary \ref{cor-3-1}. This is impossible as $\deg(u_{i+3}) \le 3$ and $u_{i+3}$ must has an extra neighbor $w_{i+3}$ with $r(w_{i+3})<r(u_{i+3})$. Suppose $2\le i \le l-4$. Then $G$ contains a subgraph (b) of the following figure with $\deg(v_{i+2})=2$ by Corollary \ref{cor-3-1}. Applying case $(3^0,3^-)$ of Lemma \ref{lem-3-1} on $P'=v_0\cdots w_{i+3}u_{i+3}v_{i+3}v_{i+4} \cdots$, we get a contradiction. Suppose $i=l-3$. Then $G$ contains a subgraph (c) of the following figure with $\deg(v_{l-1})=2$ by Corollary \ref{cor-3-1}. We have $\kappa(u_l,u_{l-2})<0$ by Theorem \ref{thm-tnc}. This completes the proof.
\begin{figure}[H]
\centering
\subfigure[]{
\includegraphics[width=0.25\textwidth]{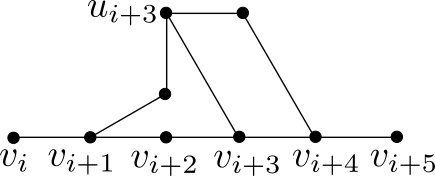}} \quad
\
\subfigure[]{
\includegraphics[width=0.25\textwidth]{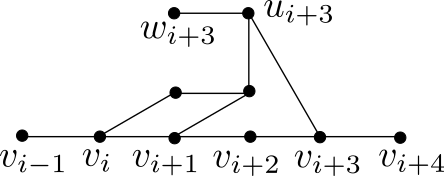}} \quad
\
\subfigure[]{
\includegraphics[width=0.25\textwidth]{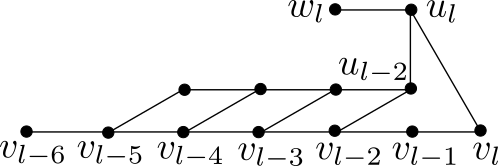}}
\caption{}
\end{figure}

\end{proof}

\section{Classification of graphs}
Let $G \in \mathcal{G}$ be a finite graph and $P=v_0v_1\cdots v_l$ be a diameter of $G$. Then $P$ is a geodesic path. We explore all the available structures of $G$ separately by the following four cases.\\
Case (i): For all $1\le i \le l-1$, $s(v_i)=2$. We discuss this case in Theorem \ref{thm-4-1}.\\
Case (ii): There exists a $v_i$ such that $s(v_i)=3^0$. We discuss this case in Theorem \ref{thm-4-2}.\\
Case (iii): There exists a $v_i$ such that $s(v_i)=3^-$. We discuss this case in Theorem \ref{thm-4-3}.\\
Case (iv): There exists a $v_i$ such that $s(v_i)=3^+$. We discuss this case in Theorem \ref{thm-4-4}.\\

We firstly introduce two lemmas.
\begin{lemma}
\label{lem-4-1}
Let $G(V,E)$ be a finite graph. Suppose $V=A\sqcup B \sqcup C_1 \sqcup C_2$ satisfying:\\
(1) $E(A,B)=E(C_1,C_2)=\emptyset$;\\
(2) The induced subgraphs $G(C_1)$ and $G(C_2)$ are complete graphs, i.e. for all $u,v \in C_i$, $u\sim v$ in $G(C_i)$, $i=1,2$;\\
(3) The induced subgraphs $G_A=G(A\sqcup C_1 \sqcup C_2)$ and $G_B=G(B \sqcup C_1 \sqcup C_2)$ are connected graphs.\\
Then $G$ contains a geodesic cycle $C$ with $|C| \ge d_{G_A}(C_1,C_2)+d_{G_B}(C_1,C_2)$.
\end{lemma}
\begin{figure}[H]
\centering
\subfigure{
\includegraphics[width=0.4\textwidth]{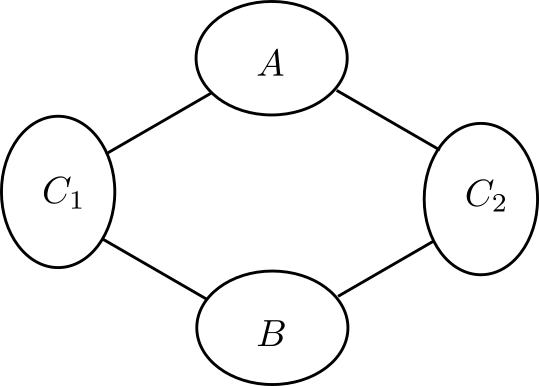}} \ \
\qquad
\subfigure{
\includegraphics[width=0.4\textwidth]{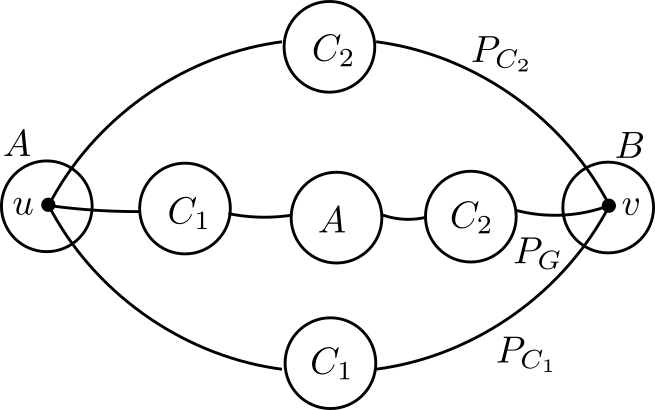}}
\caption{}
\end{figure}

\begin{proof}
Let $\mathcal{C}=$\{$W$ is a closed walk : $W=P_1P_AP_2P_B$, where $P_1$, $P_2$, $P_A$ and $P_B$ are walks in $C_1$, $C_2$, $A$ and $B$ respectively\}. Let $C \in \mathcal{C}$ with shortest length, then it is a cycle. Moreover, we can prove that $C$ is a geodesic cycle. Suppose not, there are $u,v\in C$ such that $d_G(u,v)<d_C(u,v)$. We only prove the case that $u\in A$ and $v \in B$, proofs for other cases are similar.

Let $P_G$ be a shortest walk between $u$ and $v$ in $G$ and $P_{C_i}$ be a shortest walk between $u$ and $v$ in $C$ crossing $C_i$, $i=1,\ 2$. Note that $P_{C_1}$ and $P_{C_2}$ are disjoint paths from $u$ to $v$ and $C$ is the union of them.

Next, we claim that $P_G$ has one of the following forms: $P_AP_1P_A'P_2P_B$, $P_AP_1P_BP_2P_B'$, $P_AP_2P_A'P_1P_B$, $P_AP_2P_BP_1P_B'$, $P_AP_1P_B$ or $P_AP_2P_B$. In fact, it is easy to see that $P_G\cap C_1$ is connected in $P_G$ and $|P_G\cap C_1|\le2$ as $C_1$ is complete and $P_G$ is a shortest walk. The same conclusion holds for $P_G\cap C_2$. So the claim holds.

Now, if $P_G=P_AP_1P_A'P_2P_B$, then $P_G$ and $P_{C_1}$ form a closed walk. By connecting $C_1$ part in $P_G$ and $C_1$ part in $P_{C_1}$, we get a new closed walk $C'\in \mathcal{C}$ with a shorter length. This contradicts to the fact that $C$ is a shortest walk. Proofs for other forms are similar. Then we proved $d_G(u,v)=d_C(u,v)$ for all $u,v \in C$, i.e. $C$ is a geodesic cycle. The inequality $|C| \ge d_{G_A}(C_1,C_2)+d_{G_B}(C_1,C_2)$ is obtained directly from the construction of $C$.
\end{proof}

\begin{lemma}
\label{lem-4-2}
Suppose $G\in \mathcal{G}$ contains a geodesic cycle $C$ with $|C| \ge 8$, then $|C|\ge 11$ and $G$ must be a cycle, a prism graph or a M\"obius ladder.
\end{lemma}





\begin{proof}
Suppose $|C|=8$. We assume that there exists a vertex, say $v_0$, has an extra neighbor $u_0$, otherwise $G=C$. Note that every path $P$ in $C$ with $|P|\le 5$ is geodesic in $G$.

We claim that $G$ is a prism or M\"obius ladder if $u_0 \sim v_2$. By applying local structure $(3^-,\cdot)$ on $v_1v_2v_3v_4$ in $(G,v_0)$, we get that $G$ contains (a) or (b) of the following figure as a subgraph. By symmetric, $G$ has a subgraph (c) or (d) in the following figure. By doing this repeatedly, we get that $G$ has a subgraph (e). Finally we get a prism graph or a M\"obius ladder.

\begin{figure}[H]
\centering
\subfigure[]{

\includegraphics[width=0.3\textwidth]{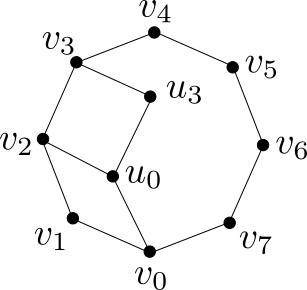}} \ \
\subfigure[]{

\includegraphics[width=0.3\textwidth]{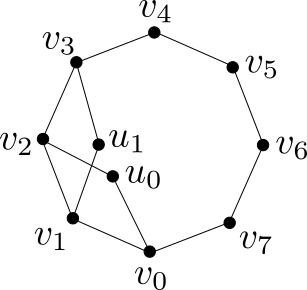}} \ \
\subfigure[]{

\includegraphics[width=0.3\textwidth]{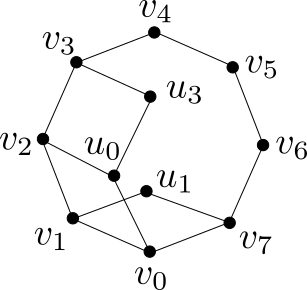}} \ \
\subfigure[]{

\includegraphics[width=0.3\textwidth]{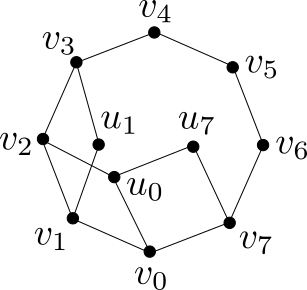}} \ \
\subfigure[]{

\includegraphics[width=0.3\textwidth]{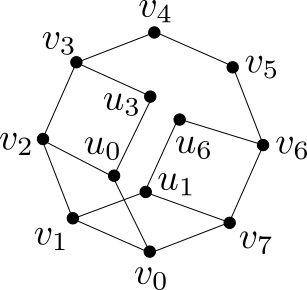}} \ \
\caption{}

\end{figure}

We claim that $u_0$ is not a neighbor of $v_1$. Suppose not, $G$ has a subgraph (a) of the following figure. Applying local structure $(3^0,\cdot)$ on $v_0v_1v_2v_3$ in $(G,v_0)$, we know that $G$ has a subgraph (b) or (c) of the following figure. This is impossible by the symmetry of the geodesic cycle. So our claim is true.
\begin{figure}[H]
\centering
\subfigure[]{
\includegraphics[width=0.3\textwidth]{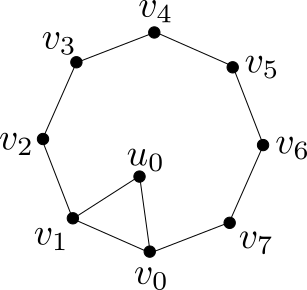}} \ \
\subfigure[]{
\includegraphics[width=0.3\textwidth]{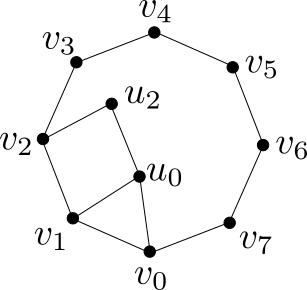}} \ \
\subfigure[]{
\includegraphics[width=0.3\textwidth]{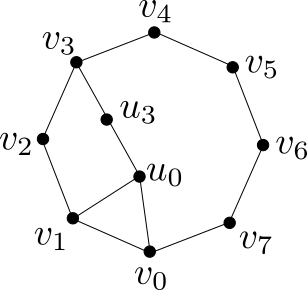}} \ \
\caption{}
\end{figure}

The previous claim also shows $u_0 \nsim v_7$. By applying local structure $(3^+, \cdot)$ on $v_7v_0v_1v_2$ in $(G,v_7)$, we know that $G$ has one of the following subgraphs.
\begin{figure}[H]
\centering
\subfigure[]{
\includegraphics[width=0.32\textwidth]{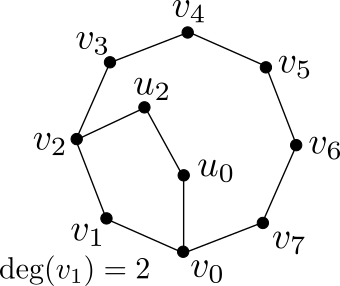}} \
\subfigure[]{
\includegraphics[width=0.3\textwidth]{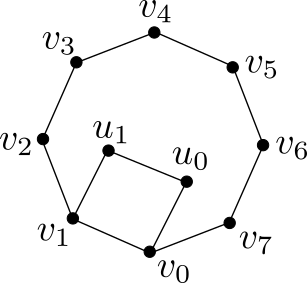}} \ \
\subfigure[]{
\includegraphics[width=0.3\textwidth]{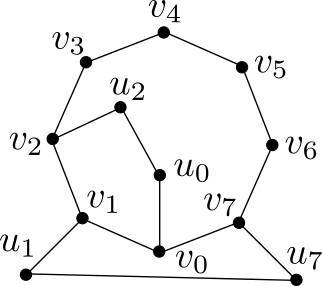}} \ \
\caption{}
\label{Fig.C6}
\end{figure}

By applying local structures on $C$ as we did previously, we get the following results. If $G$ has a subgraph (c) in Figure \ref{Fig.C6}, then $G$ is isomorphic to (a) of Figure \ref{fig-others}. Otherwise $G$ must be isometric to (b), (c) or (d) of the following figure if $G$ is not a prism graph or M\"obius ladder.
\begin{figure}[H]
\centering
\subfigure[]{
\includegraphics[width=0.27\textwidth]{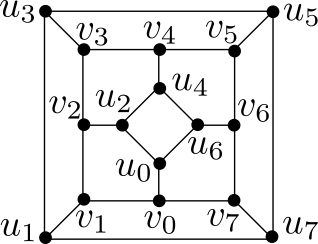}} \qquad
\subfigure[]{
\includegraphics[width=0.25\textwidth]{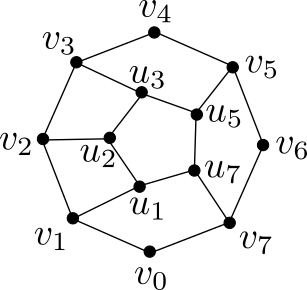}} \qquad \qquad
\subfigure[]{
\includegraphics[width=0.25\textwidth]{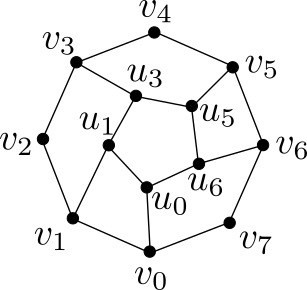}} \qquad
\subfigure[]{
\includegraphics[width=0.25\textwidth]{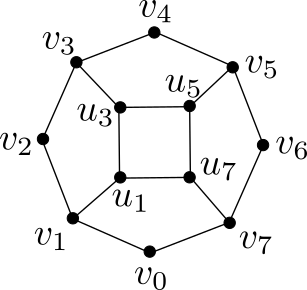}}
\caption{}
\label{fig-others}
\end{figure}

Note that the diameters of all the possible graphs we get are less than 6, thus they are not elements of $\mathcal{G}$. It follows that $|C| \ne 8$. We also have $|C| \ne 9,10$ with similar arguments. When we do the previous steps on a geodesic cycle $C$ with $|C|\ge11$, the graphs with the same forms shown in Figure \ref{fig-others} contain edges with negative curvature. This indicates that $G$ must be a cycle, prism graph or M\"obius ladder.

\end{proof}

\begin{theorem}
\label{thm-4-1}
Let $P=v_0v_1\cdots v_l$ be a diameter of $G \in \mathcal{G}$ with $l \ge 6$. If for all $1\le i\le l-1$,  $s(v_i)=2$, then $G=P$ or isometric to a cycle.
\end{theorem}
\begin{proof}
Let $G_B:=G(V\setminus \{v_0,v_1,\cdots, v_l\})$. If $v_0$ is not connected to $v_l$ via $G_B$, then $G=P$, otherwise we will get a longer diameter. If $v_0$ is connected to $v_l$ via $G_B$. By setting $C_1=\{v_0\}$ and $C_2=\{v_l\}$ and applying Lemma \ref{lem-4-1}, we get a geodesic cycle $C$ with $|C|\ge 12$. Then the theorem holds by Lemma \ref{lem-4-2}.
\end{proof}

\begin{lemma}
\label{lem-4-3}
Let $P=v_0v_1\cdots v_l$, $l \ge 6$ be a diameter of $G \in \mathcal{G}$. If there exists $3\le i \le l-3$ such that $s(v_i)=3^0$, then $s(v_j)=3^0$ for all $3\le i \le l-3$, and $G$ contains a subgraph shown in the following figure.
\begin{figure}[H]
\centering
\includegraphics[width=0.3\textwidth]{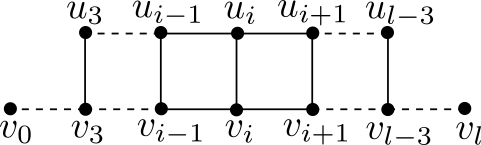}
\caption{}
\end{figure}
\end{lemma}
\begin{proof}
If $l=6$, it is trivial. Suppose $l\ge 7$ and $i+1 \le l-3$.

Firstly we prove $s(v_{i+1})=3^0$. If $s(v_{i+1})=2$, by applying local structure $(3^0,2)$ on $v_{i-1}v_iv_{i+1}v_{i+2}$ and Corollary \ref{cor-3-1}, we know that $G$ contains the following figure as a subgraph with $\deg(v_{i+1})=2$. As $r(u_i)>0$, there is a vertex $w_i\sim u_i$ with $r(w_i)<r(u_i)$. By applying local structure $(3^0,3^+)$ on $w_iu_iu_{i+2}u_{i+3}$, we get a contradiction. If $s(v_{i+1})=3^-$, by applying local structure $(3^0,3^-)$ on $v_{i-1}v_iv_{i+1}v_{i+2}$ and Corollary \ref{cor-3-1}, we also get a contradiction. It follows that $s(v_{i+1})=3^0$. By applying local structure $(3^0,3^0)$ on $v_{i-1}v_iv_{i+1}v_{i+2}$, we have $u_i \sim u_{i+1}$. If $i-1 \ge 3$, with the similar argument, we can also prove that $s(v_{i-1})=3^0$.
\begin{figure}[H]
\centering
\includegraphics[width=0.3\textwidth]{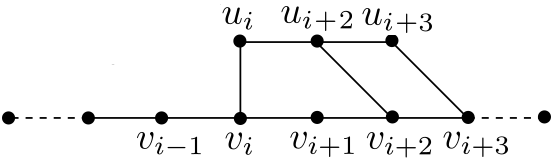}
\caption{}
\end{figure}
By taking the previous steps repeatedly, we complete the proof.
\end{proof}

\begin{theorem}
\label{thm-4-2}
Let $P=v_0v_1\cdots v_l$, $l \ge 6$, be a diameter of $G \in \mathcal{G}$. If there exists a $i \in \{1,2,\cdots,l-1\}$ such that $s(v_i)=3^0$, then $G$ is a quasi-ladder or isometric to the particular graph (d) in Figure \ref{fig-main-1}.
\end{theorem}

\begin{proof}

Let $k=\min\{i:s(v_i)=3^0,1\le i \le l-1\}$. By Lemma \ref{lem-4-3}, we have $k \in \{1,2,3,l-2,l-1\}$.

Case (i): $k=1$. $G$ contains (a) of the following figure as a subgraph. We claim that $s(v_j)=3^0$ or $3^-$ for all $2\le j \le l-3$. If $s(v_2)=2$, by applying local structure $(3^0,2)$ on $v_0v_1v_2v_3$, we have that $G$ has a subgraph (b) of the following figure. By applying local structure $(3^0,3^+)$ on $v_0u_0u_3u_4$, we get a contradiction. If $s(v_2)=3^+$, we also get a contradiction by applying $(3^0,3^+)$ on $v_0v_1v_2v_3$. So $s(v_2)=3^0$ or $3^-$.
\begin{figure}[H]
\centering
\subfigure[]{
\includegraphics[width=0.27\textwidth]{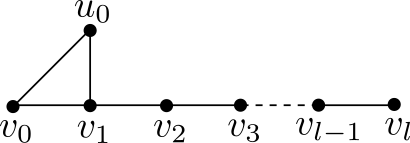}} \qquad
\subfigure[]{
\includegraphics[width=0.27\textwidth]{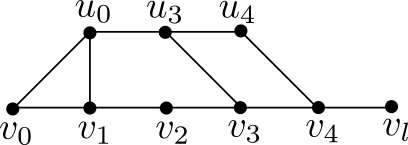}}
\caption{}
\end{figure}

If $s(v_2)=3^0$, we have $u_0 \sim u_2$ as shown in (a) of the following figure by applying local structure $(3^0,3^0)$ on $v_0v_1v_2v_3$. The previous argument tells us that $s(v_3)\ne 2,3^+$. It is not hard to verify that $s(v_3)\ne 3^-$ neither. So $s(v_3)=3^0$ and $u_2 \sim u_3$. By Lemma \ref{lem-4-3}, $G$ has a subgraph (a) of the following figure. Moreover, $\deg(v_0)=2$, otherwise $v_0$ has an extra neighbor $w_0$. Then $\deg(w_0)=1$ by Corollary \ref{cor-tnc} and $d_G(w_0,v_l)=l+1$, which is contrary to $\diam(G)=l$. If $s(v_2)=3^-$, by Corollary \ref{cor-3-1} we have $s(v_i)=3^-$ for $i\in \{2,3,\cdots,l-1\}$ and $G$ has a subgraph (b) of the following figure. Moreover, $\deg_G(u_0)=2$ and $u_0v_1v_2\cdots v_l$ is also a diameter of $G$. We can redraw (b) as shown in (c) of the following figure, which contains (a) as a subgraph. So we only need to consider the situation that $G$ contains a subgraph (a).

\begin{figure}[H]
\centering
\subfigure[]{
\includegraphics[width=0.27\textwidth]{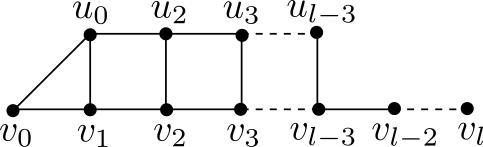}} \ \
\subfigure[]{
\includegraphics[width=0.27\textwidth]{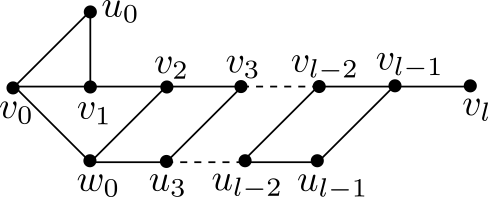}} \ \
\subfigure[]{
\includegraphics[width=0.27\textwidth]{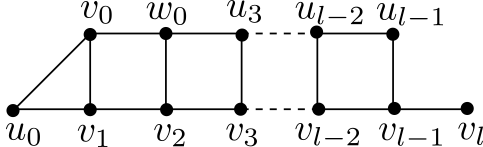}} \ \
\caption{}
\label{}
\end{figure}

As we see, the left side of $G$ is fixed, we start to study the structure of the right side. If $\deg(v_{l-2})=2$, $G$ contains a subgraph (a) of the following figure by applying local structure $(3^0,2)$ on $v_{l-4}v_{l-3}v_{l-2}v_{l-1}$. If $\deg(v_{l-2})=3$, we have $s(v_{l-2})=3^0$ and $u_{l-3}\sim u_{l-2}$. With similar arguments, we classify all possible structures of the right side of $G$ as shown in the following figure.
\begin{figure}[H]
\centering
\subfigure[]{
\includegraphics[width=0.27\textwidth]{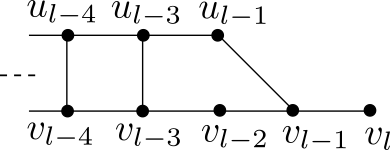}} \ \
\subfigure[]{
\includegraphics[width=0.27\textwidth]{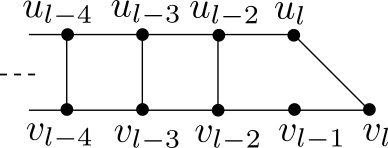}} \ \
\subfigure[]{
\includegraphics[width=0.27\textwidth]{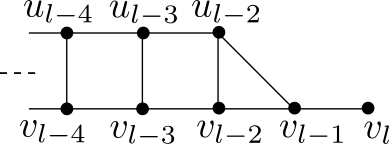}} \ \
\subfigure[]{
\includegraphics[width=0.27\textwidth]{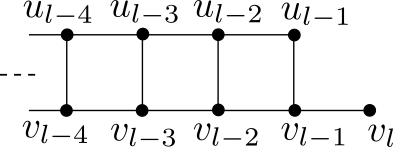}} \ \
\subfigure[]{
\includegraphics[width=0.27\textwidth]{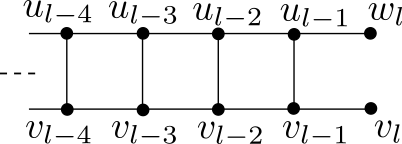}} \ \
\subfigure[]{
\includegraphics[width=0.27\textwidth]{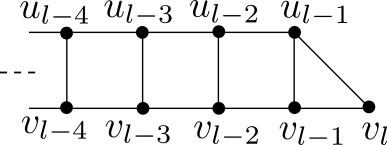}} \ \
\subfigure[]{
\includegraphics[width=0.27\textwidth]{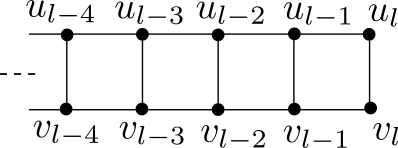}} \ \
\caption{}
\label{fig-r1}
\end{figure}

Case (ii): $k=2$. By applying $(\cdot,3^0)$ on $v_0v_1v_2v_3$, $G$ contains a subgraph (a), (b) or (c) shown in the following figure. Obviously situation (c) is contained in (b) as $u_0u_2v_2v_3\cdots$ in (c) is also a diameter. We only need to consider the situation that $G$ has a subgraph (a) or (b). It is easy to verify that $\deg(v_0)=1$ in (a) and $\deg(v_0)=\deg(v_1)=\deg(u_0)=2$ in (b). So the left-hand side is fixed. The right-hand side of $G$ is the same with that in case (i).
\begin{figure}[H]
\centering
\subfigure[]{
\includegraphics[width=0.3\textwidth]{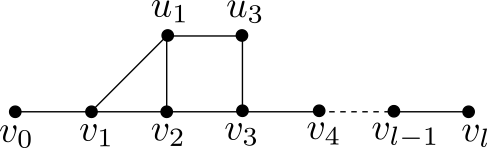}} \ \
\subfigure[]{
\includegraphics[width=0.3\textwidth]{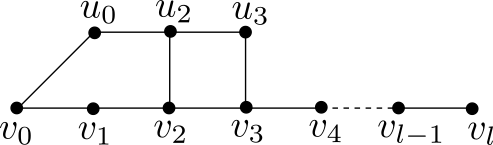}} \ \
\subfigure[]{
\includegraphics[width=0.3\textwidth]{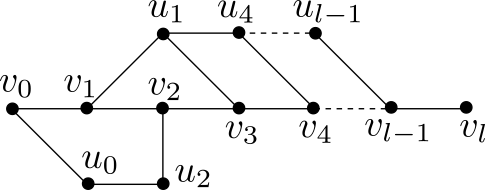}} \ \
\caption{}
\label{}
\end{figure}

Case (iii): $k=3$. If $l\ge 7$, the left side of $G$ is shown as following. The available structures of the right side are the same with that in case (i). If $l=6$, $G$ has an additional available structure shown in (d) of Figure \ref{fig-main-1}.

\begin{figure}[H]
\centering
\includegraphics[width=0.3\textwidth]{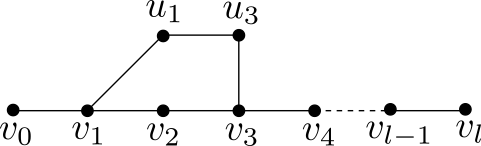}
\caption{}
\end{figure}

Case (iv): $k=l-2$. It is easy to verify this case can not happen.

Case (v): $k=l-1$. $G$ has a subgraph (a) or (b) shown in the following figure. As $v_l v_{l-1} \cdots v_0$ is also a diameter, graphs in this case have been discussed in the previous cases.
\begin{figure}[H]
\centering
\subfigure[]{
\includegraphics[width=0.3\textwidth]{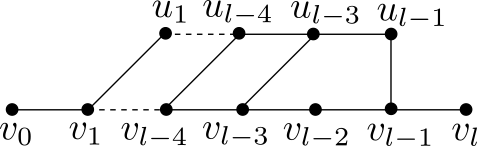}} \qquad
\subfigure[]{
\includegraphics[width=0.3\textwidth]{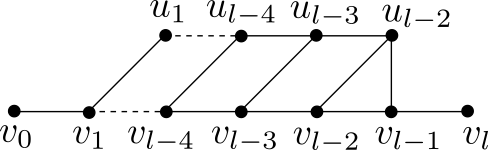}}
\caption{}
\end{figure}

By previous discussions, we prove that $G$ must be a quasi-ladder graph with the following structures or isometric to the particular graph (d) in Figure \ref{fig-main-1}.
\begin{figure}[H]
\centering
\subfigure[Quasi-ladder]{
\includegraphics[width=0.3\textwidth]{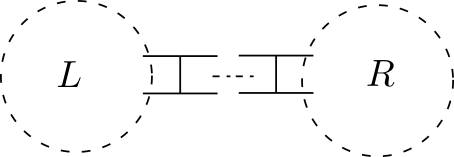}} \qquad
\subfigure[Forms of L]{
\includegraphics[width=0.2\textwidth]{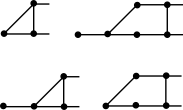}} \qquad
\subfigure[Forms of R]{
\includegraphics[width=0.3\textwidth]{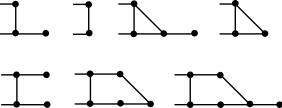}}
\caption{}
\label{Fig.main2}
\end{figure}
\end{proof}

\begin{theorem}
\label{thm-4-3}
Let $P=v_0v_1\cdots v_l$, $l \ge 6$, be a diameter of $G \in \mathcal{G}$. If for all $1\le i \le l-1, s(v_i)\ne 3^0,$ and there exists $v_i$ such that $s(v_i)=3^-$, then $G$ is a prism graph, a M\"obius ladder or a quasi-ladder.
\end{theorem}
\begin{proof}
Let $k=\min\{i: 1\le i \le l-1\ ,\ s(v_i)=3^-\}$, then $k\ge 2$.

Case (i): If $k=2$, $G$ has a subgraph (a) of the following figure. By Corollary \ref{cor-3-1}, $G$ has a subgraph (b) or (c) of the following figure. By switching $u_0$ and $v_1$ in (b), we get (b) is the same with (c). So we only consider the case that $G$ contains a subgraph (b).
\begin{figure}[H]
\centering
\subfigure[]{
\includegraphics[width=0.3\textwidth]{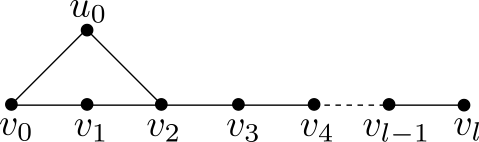}} \ \
\subfigure[]{
\includegraphics[width=0.3\textwidth]{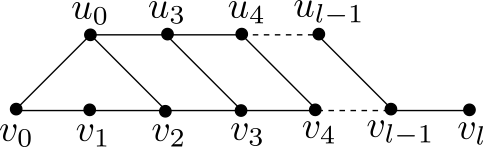}} \ \
\subfigure[]{
\includegraphics[width=0.3\textwidth]{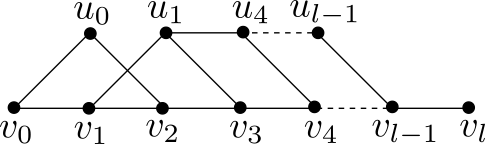}} \ \
\caption{}
\end{figure}

Let $C_1=G(\{v_0,v_1\})$, $G_A=G(\{u_0,u_3,\cdots,u_{l-2},v_2,v_3\cdots,v_{l-2}\})$, $C_2=G(\{v_{l-1},u_{l-1}\})$ and $G_B=G(V\setminus\{u_0,u_3,\cdots,u_{l-1},v_0,v_1\cdots,v_{l}\})$. If $C_1$ is connected to $C_2$ via $G_B$, we get a geodesic cycle $C$ with $|C| \ge 8$ by Lemma \ref{lem-4-1} and $G$ must be a prism graph or M\"obius ladder by Lemma \ref{lem-4-2}. In the following, we suppose that  $C_1$ is not connected to $C_2$ via $G_B$. With a similar argument in the proof of Theorem \ref{thm-4-2}, the left side of $G$ has one of the following forms.
\begin{figure}[H]
\centering
\subfigure[]{
\includegraphics[width=0.3\textwidth]{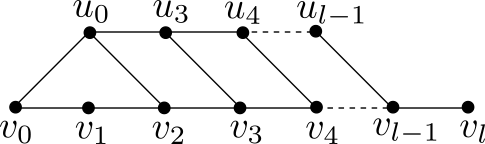}} \qquad
\subfigure[]{
\includegraphics[width=0.3\textwidth]{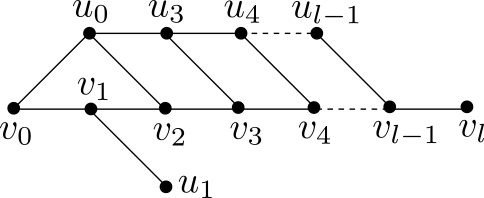}}
\caption{}
\label{fig-l3}
\end{figure}

We start to deal with the right-hand side. If $\deg(u_{l-1})=2$, we have $\deg(v_l)=1$ by Corollary \ref{cor-tnc} and the structure of the right-hand side is (a) of the following figure. If $u_{l-1}$ has an extra neighbor $w_{l-1}$, we have $r(u_{l-1}) < r(w_{l-1})$. Otherwise, $w_{l-1}$ has an extra neighbor $w'$ and $\kappa(w',w_{l-1})<0$ by Corollary \ref{cor-tnc}. If $\deg(w_{l-1})=1$, we have $\deg(v_l)=1$ and the right-hand side structure is shown in (b) of the following figure. With similar arguments, we find all possible forms of the right side of $G$ as shown in the following figure.

\begin{figure}[H]
\centering
\subfigure[]{
\includegraphics[width=0.3\textwidth]{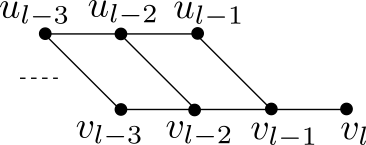}} \ \
\subfigure[]{
\includegraphics[width=0.3\textwidth]{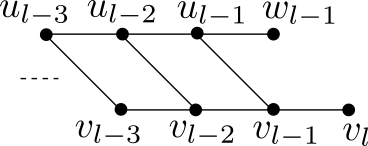}} \ \
\subfigure[]{
\includegraphics[width=0.3\textwidth]{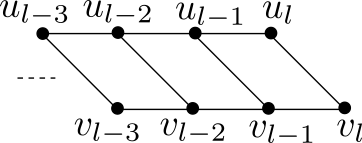}} \ \
\subfigure[]{
\includegraphics[width=0.3\textwidth]{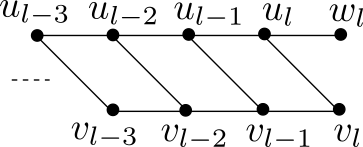}} \ \
\subfigure[]{
\includegraphics[width=0.3\textwidth]{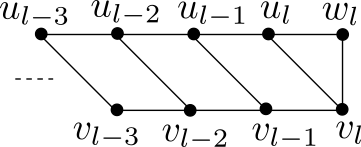}} \ \
\caption{}
\label{fig-r2}
\end{figure}


Case (ii): $k=3$. $G$ has a subgraph (a) of the following figure. The left side of $G$ must be (a) or (b) of the following figure and all possible forms of the right side have been shown in Figure \ref{fig-r2}.
\begin{figure}[H]
\centering
\subfigure[]{
\includegraphics[width=0.35\textwidth]{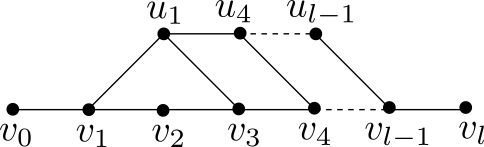}} \qquad
\subfigure[]{
\includegraphics[width=0.35\textwidth]{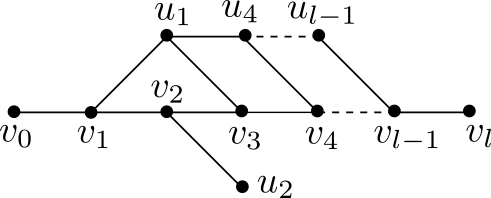}}
\caption{}
\label{fig-l4}
\end{figure}

Case (iii): $4\le k \le l-2$. $G$ has a subgraph shown in the following figure which is the same with case (ii) by switching $u_{i-2}$ and $v_{i-1}$, $4\le i \le k$.
\begin{figure}[H]
\centering
\includegraphics[width=0.45\textwidth]{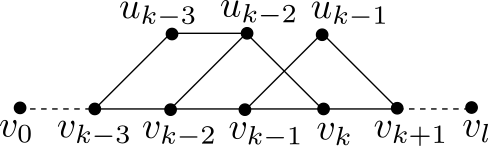}
\caption{}
\end{figure}

Case (iv): $j=l-1$. $G$ contains a subgraph shown in the following figure which is the same with (a) of Figure \ref{fig-l4} if $u_{l-2}\sim v_l$, or (b) of Figure \ref{fig-l4} if $u_{l-2} \nsim v_l$ as $P'=v_lv_{l-1}\cdots v_0$ is also a diameter.
\begin{figure}[H]
\centering
\includegraphics[width=0.45\textwidth]{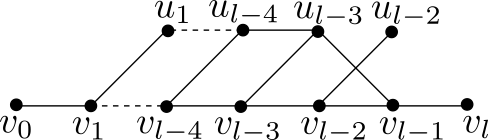}
\caption{}
\end{figure}

By previous discussions, we prove that if $G$ is not a prism graph or a M\"{o}bius ladder, then it has one of the following forms, which are isometric to graphs in Figure \ref{fig-thm3-quasi-ladder}.

\begin{figure}[H]
\centering
\subfigure[]{
\includegraphics[width=0.3\textwidth]{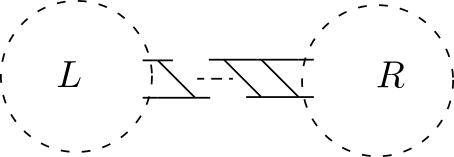}} \ \ \
\subfigure[Forms of L]{
\includegraphics[width=0.2\textwidth]{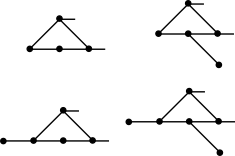}} \ \ \
\subfigure[Forms of R]{
\includegraphics[width=0.3\textwidth]{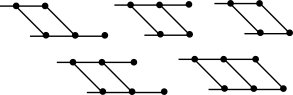}}
\caption{}
\end{figure}

\begin{figure}[H]
\centering
\subfigure[Quasi-ladder]{
\includegraphics[width=0.3\textwidth]{g19}} \ \ \
\subfigure[Forms of L]{
\includegraphics[width=0.2\textwidth]{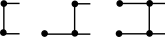}} \ \ \ \
\subfigure[Forms of R]{
\includegraphics[width=0.3\textwidth]{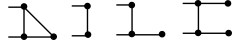}}
\caption{}
\label{fig-thm3-quasi-ladder}
\end{figure}

\end{proof}

\begin{theorem}
\label{thm-4-4}
Let $P=v_0v_1\cdots v_l$, $l \ge 6$, be a diameter of $G \in \mathcal{G}$. Suppose that for all  $1\le i \le l-1$, $s(v_i)\ne 3^0, s(v_i)\ne 3^-$, and there exists $s(v_i)=3^+$, then $G$ is a prism graph, a M\"obius ladder or a quasi-ladder.
\end{theorem}
\begin{proof}
In fact we will prove that graphs considered in this theorem have been discussed in the previous theorems. Let $k=\max\{i: 1\le i \le l-1,\ s(v_i)=3^+\}$, then $k = l-2$ or $l-1$.

Case (i): $k=l-2$. Then $\deg(v_{l-1})=2$ and $G$ has a subgraph shown in figure (a). By applying local structure $(3^+,2)$ on $v_{l-3}v_{l-2}v_{l-1}v_{l}$, we get that $G$ has a subgraph (b) or (c). We have discussed those graphs by changing $P=v_0v_1\cdots v_l$ to $P'=v_l v_{l-1}\cdots v_0$.
\begin{figure}[H]
\centering
\subfigure[]{
\includegraphics[width=0.3\textwidth]{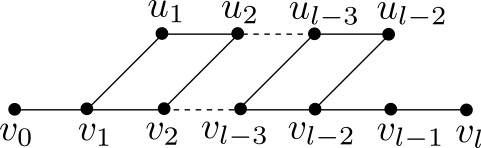}} \ \
\subfigure[]{
\includegraphics[width=0.3\textwidth]{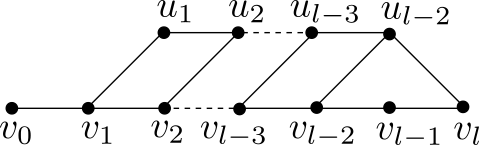}} \ \
\subfigure[]{
\includegraphics[width=0.3\textwidth]{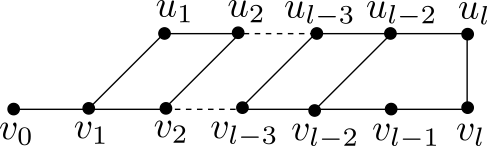}}
\caption{}
\end{figure}

Case (ii): $k=l-1$. $G$ contains one of the followings as a subgraph, which have been discussed in the previous theorems by changing $P=v_0v_1\cdots v_l$ to $P'=v_l v_{l-1}\cdots v_0$.
\begin{figure}[H]
\centering
\subfigure[]{
\includegraphics[width=0.32\textwidth]{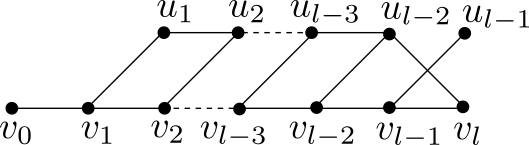}} \ \
\subfigure[]{
\includegraphics[width=0.32\textwidth]{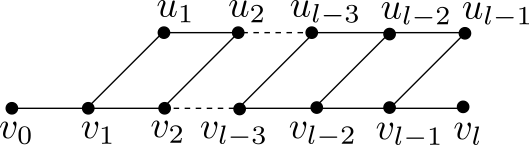}}
\caption{}
\end{figure}

\end{proof}

Now, we give the proofs of Theorem \ref{thm-main-finite} and \ref{thm-main-infty}.
\begin{proof}[Proof of Theorem \ref{thm-main-finite}]
By applying Theorem \ref{thm-4-1}, \ref{thm-4-2}, \ref{thm-4-3} and \ref{thm-4-4}, we obtain the conclusion.
\end{proof}
\begin{proof}[Proof of Theorem \ref{thm-main-infty}]
As $G$ is an infinite graph, we can find an infinite geodesic path $P$ which is isometric to $\Z$ or $\Z^+$. By applying the arguments used for finite graphs, we get the conclusion.
\end{proof}

\section{Applications}
In this section, as an application of the main results, we establish the following strong Liouville property.
\begin{theorem}\label{liouville}
Let $G \in \mathcal{G}$ be an infinite graph, then all positive harmonic functions on $G$ are constant.
\end{theorem}
\begin{proof}
By Theorem \ref{thm-main-infty}, $G$ is a line, half line, infinite ladder or infinite quasi-ladder. As line and infinite ladder are finitely generated Cayley graphs of Abelian groups, all positive harmonic functions on them are constant; see \cite{MR119041}. Besides, Hua and M\"unch also proved all positive harmonic functions are constant on graphs with two ends and nonnegative Ricci curvature, which provides another proof for the strong Liouville property of line and infinite ladder; see \cite[Theorem 5.9]{hua2021salami}.  For a half line, there is a stronger conclusion. In fact all harmonic functions on a half line are constant. So, we only consider the quasi-ladder case.


Let $f$ be a positive harmonic function on $G$. If $f$ is not a constant, there exist $i\ge 0$, such that $f(x_i)\ne f(y_i)$. Without loss of generality, we assume $i=0$. By choosing suitable constants $a$ and $b$  such that $g:=af+b$ satisfies $g(x_0) =1$ and $g(y_0)=0$, we get a harmonic function $g$ which is bounded from below or above.
\begin{figure}[H]
\centering
\includegraphics[width=0.45\textwidth]{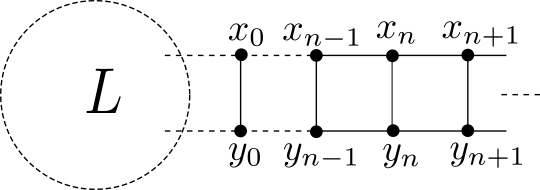}
\end{figure}

For any $n\in \mathbb{Z}^+$, we have
\[
g(x_n) = \frac13(g(x_{n-1})+g(x_{n+1})+g(y_n)),
\]
\[
g(y_n) = \frac13(g(y_{n-1})+g(y_{n+1})+g(x_n)).
\]
Let $z_n = g(x_n)-g_(y_n)$, $h_n=g(x_n)+g(y_n)$. We have $4z_n = z_{n-1}+z_{n+1}$, $2h_n = h_{n-1}+h_{n+1}$. On the one hand, it is easy to check that $h_n = O(n)$. On the other hand, $z_n \ge z_{n-1}$ by the maximum principle. Thus $z_{n+1} \ge 3z_n$, which means that $z_n \ge C\cdot3^n$ for some constant $C>0$. It follows that $g(y_n) \to -\infty$ and $g(x_n) \to +\infty$ as $n \to +\infty$. We get a contradiction as $g$ is bounded from below or above.
\end{proof}

\section*{Acknowledgments}
The authors are grateful to Prof. Bobo Hua for his guidance and support. The authors would also like to thank   Florentin M\"unch and David Cushing for their helpful advice.

\bibliographystyle{alpha}
\bibliography{ms}

\end{document}